\documentclass[reqno, 11pt]{amsart}

\pdfoutput=1
\usepackage{graphicx}
\usepackage{amssymb, amsmath, amsthm}
 \usepackage[bookmarks=false]{hyperref}
 \usepackage{mathscinet}

\allowdisplaybreaks

\makeatletter
\renewcommand{\@biblabel}[1]{#1.} 
\makeatother

\newtheorem{Theorem}{Theorem}[section]

\newtheorem{Corollary}[Theorem]{Corollary} 
\newtheorem{rem1}[Theorem]{Remark} 
\newenvironment{Remark}{\begin{rem1}\normalfont}{\end{rem1}}

\newtheorem*{rem}{Remark} 
\newenvironment{Remark*}{\begin{rem}\normalfont}{\end{rem}}

\newtheorem*{rems}{Remarks} 

\newtheorem*{sol}{Solution}

\newtheorem*{nt}{Notes}

\makeatletter
\renewcommand{\@biblabel}[1]{#1.}
\makeatother

\newcommand{\suml}{\sum\limits}

\newcommand{\sumvec}[2]{#1_1+\cdots + #1_#2} 
\newcommand{\qelementary}[2]{q^{e_2 ({\mathbf #1})}}  
\newcommand{\htqelementary}[2]{q^{#2e_2(\mathbf #1) }}  

\newcommand\sumj{{\left| \mathbf {j} \right|}}
\newcommand\sumjcurl{{\left| \mathbf {\widetilde{j}} \right|}}

\newcommand\sumvecl{{\left| \mathbf {l} \right|}}
\newcommand\sumk{{\left| \mathbf {k} \right|}}
\newcommand\sumkcurl{{\left| \mathbf {\widetilde{k}} \right|}}

\newcommand{\y}{{\mathbf y}}

\renewcommand{\k}{{\mathbf k}}
\renewcommand{\j}{{\mathbf j}}

 \newcommand{\h}{\mathbf h}
 
 \newcommand{\x}{\mathbf x}

\newcommand{\kvec}{{k_1,\dots,k_n}}

\newcommand{\ban}[1]{\hbox{$\boldsymbol A_{\boldsymbol #1}$}}

\newcommand{\qrfac}[2]{{\left({#1}; q \right)_{#2}}} 

\newcommand{\pqrfac}[3]{{\left({#1};#3\right)_{#2}}}

\newcommand{\triprod}[1]{\prod\limits_{1\le r < s \le #1}}

\newcommand{\sqprod}[1]{\prod\limits_{r, s =1}^{#1}} 

\newcommand{\smallprod}[1]{\prod\limits_{r =1}^{#1}} 

\newcommand{\xover}[1]{{#1_{r}}/{#1_{s}}}

\newcommand{\powerq}[2]{q^{\suml_{r=1}^{#2} (r-1)#1_r }} 

\newcommand{\hpowerq}[2]{q^{h\suml_{r=1}^{#2} (r-1)#1_r }} 

\newcommand{\htpowerq}[3]{q^{#3\suml_{r=1}^{#2} (r-1)#1_r }} 

\newcommand{\vandermonde}[3]{\triprod{#3} \! 
\frac{1-q^{#2_r-#2_s} \xover {#1} }{1-\xover{#1}}
}

\newcommand{\hvandermonde}[3]{\triprod{#3} \! 
\frac{1- q^{h(#2_r-#2_s)} \xover {#1}}{1-\xover{#1}}
}

\newcommand{\htvandermonde}[4]{\triprod{#3} \! 
\frac{1-q^{#4(#2_r-#2_s)}\xover {#1} }{1-\xover{#1}}
}

\newcommand{\vandespecial}[2]{\triprod{#2} \! 
\frac{1- q^{#2(#1_r-#1_s) +r-s}}{1-q^{r-s}}
}
\newcommand{\vandespecialt}[3]{\triprod{#2} \!
\frac{1- q^{#3(#2(#1_r-#1_s) +r-s)}}{1-q^{#3(r-s)}}
}

\newcommand{\multisum}[2]{\underset{r=1, \dots, #2}{\sum\limits_{#1_r\geq 0}}} 

\numberwithin{equation}{section}

\begin{document} 
\title{Heine's method and $\ban n$ to $\ban m$ transformation formulas}

\author[Gaurav Bhatnagar]{Gaurav Bhatnagar*
}
\thanks{*Research supported in part by the Austrian Science Fund (FWF): Project F 50 SFB (Algorithmic and Enumerative Combinatorics). 
Part of this work was done while the author was a Visiting Scientist at the Indian Statistical Institute, Delhi Centre, 
7 S.~J.~S.~Sansanwal Marg, Delhi 110016.}
\address{Fakult\"at f\"ur Mathematik,  Universit\"at Wien \\
Oskar-Morgenstern-Platz 1, 1090 Wien, Austria.}
\email{bhatnagarg@gmail.com}

\date{\today}




\begin{abstract}
We apply Heine's method---the key idea Heine used in 1846 to derive his famous transformation formula for $_2\phi_1$ series---to multiple basic series over the root system of type $A$. In the classical case, this leads to a bibasic extension of Heine's formula, which was implicit in a paper of Andrews which he wrote in 1966. As special cases, we recover extensions of many of Ramanujan's $_2\phi_1$ transformations.  In addition, we extend previous work of the author regarding a bibasic extension of Andrews'  $q$-Lauricella function, and show how to obtain very general transformation formulas of this type. The results obtained include transformations of an $n$-fold sum into an $m$-fold sum. 
\end{abstract}

\maketitle 


\section{Introduction: Heine's method}
Heine's method, so named by Andrews and Berndt \cite{AB2009}, was what Heine used to obtain his celebrated transformation formula between two basic hypergeometric series. Using Heine's idea itself, we can extend Heine's identity to a bibasic transformation formula \cite{GB2017}. This was used to provide a unified treatment to many of Ramanujan's $_2\phi_1$ transformations presented in \cite[Ch.\ 1]{AB2009}. The objective of this paper is to apply Heine's idea in the context of multiple basic hypergeometric series over the root systems of type~A. 

It is useful to recall Heine's method. We need the notation of $q$-rising factorials from Gasper and Rahman \cite{GR90}. The $q$-rising factorial is defined as $\qrfac{A}{0} :=1$, and when $k$ is a positive integer, 
\begin{equation*}
\qrfac{A}{k} := (1-A)(1-Aq)\cdots (1-Aq^{k-1}).
\end{equation*}
The parameter $q$ is called the \lq base'.  The infinite $q$-rising factorial is defined, for $|q|<1$, as
\begin{equation*}
\qrfac{A}{\infty} := \prod_{r=0}^{\infty} (1-Aq^r).
\end{equation*}
Observe that, for  $|q|<1$ \cite[Eq.\ (I.5)]{GR90},
\begin{equation}\label{elementary1}
\qrfac{A}{k} = \frac{\qrfac{A}{\infty}}{\qrfac{Aq^k}{\infty}},
\end{equation}
an identity that is used to define $q$-rising factorials when $k$ is an arbitrary complex number. 
The most fundamental of the $q$-summation theorems is the $q$-binomial theorem
\cite[Eq.\ (1.3.2)]{GR90}:
For $|z|<1$, $|q|<1$
\begin{equation}\label{q-bin}
\frac{\qrfac{az}{\infty}}{\qrfac{z}{\infty}} = 
\sum_{k=0}^{\infty} 
\frac{\qrfac{a}{k}}{\qrfac{q}{k}} z^k.
\end{equation}

Heine's transformation formula \cite[Eq.\ (1.4.1)]{GR90} is as follows.  For $|z|<1$, $|b|<1$:
\begin{equation}\label{heine}
\sum_{k=0}^{\infty} \frac{\qrfac{a}{k} \qrfac{b}{k} }{\qrfac{c}{k}\qrfac{q}{k}} z^k =
\frac{\qrfac{b}{\infty} \qrfac{az}{\infty}}{\qrfac{c}{\infty}\qrfac{z}{\infty}}
\sum_{j=0}^{\infty} \frac{\qrfac{c/b}{j} \qrfac{z}{j} }{\qrfac{az}{j}\qrfac{q}{j}} b^j.
\end{equation}
Heine's \cite{EH1846} proof of his transformation formula uses the $q$-Binomial theorem, and is as follows. 
\begin{align}
\sum_{k=0}^{\infty} \frac{\qrfac{a}{k} \qrfac{b}{k} }{\qrfac{q}{k}\qrfac{c}{k}} z^k 
&=
\frac{\qrfac{b}{\infty} }{\qrfac{c}{\infty}}
\sum_{k=0}^{\infty} \frac{\qrfac{a}{k} }{\qrfac{q}{k}}z^k 
\frac{\qrfac{cq^k}{\infty} }{\qrfac{bq^k}{\infty}} 
\quad \text{ (using \eqref{elementary1})}
 \cr
&= \frac{\qrfac{b}{\infty} }{\qrfac{c}{\infty}}
\sum_{k=0}^{\infty} \frac{\qrfac{a}{k} }{\qrfac{q}{k}}z^k 
\sum_{j=0}^{\infty} \frac{\qrfac{c/b}{j}  }{\qrfac{q}{j}} (bq^k)^j
\quad \text{ (using \eqref{q-bin})}
\cr
 &= \frac{\qrfac{b}{\infty} }{\qrfac{c}{\infty}}
\sum_{j=0}^{\infty} \frac{\qrfac{c/b}{j}  }{\qrfac{q}{j}} b^j
\sum_{k=0}^{\infty} \frac{\qrfac{a}{k} }{\qrfac{q}{k}}(zq^j)^k 
 \cr
 &= \frac{\qrfac{b}{\infty} }{\qrfac{c}{\infty}}
\sum_{j=0}^{\infty} \frac{\qrfac{c/b}{j}  }{\qrfac{q}{j}} b^j
\frac{\qrfac{azq^j}{\infty} }{\qrfac{zq^j}{\infty}}
\quad \text{ (using \eqref{q-bin} again)}
 \cr
 &= \frac{\qrfac{b}{\infty}\qrfac{az}{\infty} }{\qrfac{c}{\infty}\qrfac{z}{\infty}}
\sum_{j=0}^{\infty} \frac{\qrfac{c/b}{j} \qrfac{z}{j} }{\qrfac{q}{j}\qrfac{az}{j}} b^j
\quad \text{ (using \eqref{elementary1} again).}
\notag
\end{align}

There are several mild variations of Heine's original idea that we will use in this paper. Virtually the same proof works with more general bases $q^h$ and $q^t$, where $h$ and $t$ are complex numbers. This leads to a transformation formula that follows from a very general identity of Andrews \cite{An1966a},  but was stated explicitly only in  \cite[Eq.\ (2.4)]{GB2017}.  
\begin{equation}\label{GB-sym-Heine}
\sum_{k=0}^{\infty} \frac{\pqrfac{a}{k}{q^h}}{\pqrfac{q^h}{k}{q^h}}\frac { \pqrfac{w}{hk}{q^t} }{\pqrfac{bw}{hk}{q^t}} z^k 
=
\frac{\pqrfac{w}{\infty}{q^t}\pqrfac{az}{\infty}{q^h} }{\pqrfac{bw}{\infty}{q^t} \pqrfac{z}{\infty}{q^h}}
\sum_{j=0}^{\infty} \frac{\pqrfac{b}{j}{q^t}} {\pqrfac{q^t}{j}{q^t} }\frac {\pqrfac{z}{tj}{q^h}}{\pqrfac{az}{tj}{q^h}} w^j.
\end{equation}
Here we require $|q^h|<1$, $|q^t|<1$ and $|q^{ht}|<1$ for the $q$-rising factorials to be defined, and 
$|w|<1$ and $|z|<1$ for absolute convergence of the two series. For details of how to test for convergence, refer to \cite{GB2017}. This reduces to Heine's transformation when $h=1=t$, and $(w,  b)$ is replaced by $(b, c/b)$.  

The second variation of Heine's ideas come from using multiple series extensions of the $q$-binomial theorem, due to Milne \cite{Milne1997}, Milne and Lilly \cite{ML1995}, Gustafson and Krattenthaler \cite{GK1996}, Kajihara \cite{Ka2012}, and one implicit in the work of the author with Schlosser \cite{BS2018a}.  The series are all of the form
$$
\sum\limits_{\substack{{k_r\geq 0 } \\  
{r =1,2,\dots, n}}} S(\k)
$$
where $\k=(\kvec )$, and  $k_1, k_2, \dots, k_n$ are non-negative integers. The positive integer $n$ is called the {\em dimension} of the sum. When $n=1$, we refer to the corresponding identity as {\em classical}. 
We use the notation $\sumk:= \sumvec{k}{n}$ for the sum of components of $\k$.
These type of series are recognized by the presence of the so-called  \lq\lq Vandermonde factor" of type $A$, namely
$$\vandermonde{x}{k}{n}.$$ Sometimes this factor is hidden in the series. 
For example, here is an identity we discovered during the course of our study. For $|q|<1$, $|q^t|<1$, we have
\begin{align}
\pqrfac{(-aq)^{n}}{\infty}{q^{n}} 
\multisum{j}{m} & \Bigg( \vandespecial{j}{m}
\notag \\
& \times
\smallprod{m} \frac{1}{\pqrfac{q^{r}}{mj_r}{q} } \cdot
  \frac{b^{m\sumj} }{\pqrfac{(-aq)^{n}}{tm\sumj}{q^{n}}}
  \notag \\
& \times
q^{2m\sum\limits_{r=1}^m (r-1)j_r - m(m-1)\sumj}
q^{\sum\limits_{r=1}^m {mj_r+1\choose 2}}
\Bigg) \notag \\
= 
\pqrfac{(-bq)^{m}}{\infty}{q^{m}} \multisum{k}{n} \Bigg( & \vandespecial{k}{n}
\cr
&\times \smallprod{n} \frac{1}{\pqrfac{q^{r}}{nk_r}{q} }\cdot
 \frac{a^{n\sumk} }{\pqrfac{(-bq)^{m}}{tn\sumk}{q^{m}}}
 \cr
&\times
q^{2n\sum\limits_{r=1}^n (r-1)k_r -n(n-1)\sumk}
q^{\sum\limits_{r=1}^n {nk_r+1\choose 2}}
\Bigg),
\label{an-m-GB-1.4.17b}
\end{align}
obtained as a special case of \eqref{an-m-GB2016-eq26-b}  below. 
Compare this with Ramanujan's identity \cite[Entry 1.4.17]{AB2009}, to which \eqref{an-m-GB-1.4.17b} reduces when $n$ and $m$ are $1$:
\begin{align}\label{1.4.17}
\pqrfac{-aq}{\infty}{q} 
\sum_{j=0}^{\infty} 
\frac {b^j q^{{j+1\choose 2}} }{\pqrfac{q}{j}{q} \pqrfac{-aq}{tj}{q}} 
= 
\pqrfac{-bq}{\infty}{q} \sum_{k=0}^{\infty} 
\frac{ a^k q^{{k+1\choose 2}} }{\pqrfac{q}{k}{q} \pqrfac{-bq}{tk}{q}}.
\end{align}

As we will see, Heine's method, though more than 150 years old, is still  surprisingly useful. In \S \ref{sec:GB-Heine}, we provide a few examples of multivariable transformation formulas generalizing \eqref{GB-sym-Heine}. A feature of these formulas is that they transform an $n$-dimensional sum into an $m$-dimensional sum. This is followed by some examples of multiple series extensions of Ramanujan's transformations in \S \ref{sec:ramanujan-2p1}. In the rest of the paper we explore some other variations of Heine's method, leading eventually to a master theorem describing such results in \S \ref{sec:qlauricella}. 

\section{Heine's method: The bibasic Heine transformation formula}\label{sec:GB-Heine}
In this section, we will give four multiple series extensions of \eqref{GB-sym-Heine}, using different $A_n$ extensions of the $q$-binomial theorem. Our intention is to illustrate the approach, not to provide a comprehensive list. There are six multiple $q$-binomial theorems of this kind that can be combined together to give 21 such results. 

We begin by using a theorem of Milne and Lilly \cite[Th.\ 4.7]{ML1995}:
\begin{align}
\smallprod{n} \frac{\qrfac{a_r z /x_r}{\infty}}{\qrfac{z/x_r}{\infty}}  =\multisum{k}{n}&  \Bigg(  \vandermonde{x}{k}{n} 
\sqprod{n} \frac{\qrfac{a_s\xover{x}}{k_r} }{\qrfac{q\xover{x}}{k_r} } \notag \\
&\times 
z^{\sumk} \powerq{k}{n}   \qelementary{k}{n} \smallprod{n} x_r^{-k_r}
\Bigg).
\label{an-qbin2}
\end{align}
The variables $a_1, \dots , a_n$,   $x_1,\dots,x_n$, and  $z$ are indeterminate,  and are such that the terms in the sum are well-defined. We require $|q|<1$ for convergence of the infinite products. In addition, for convergence  of the multiple series, we require  $|z/x_r|<1$, for $r=1, 2, \dots, n$. Here we use the notation $e_2(\k)$ for the elementary symmetric function of degree $2$ in the variables $\k=(k_1,k_2,\dots,k_n)$. It is given by
$$e_2(\k) = {\sumk \choose 2} -\sum_{r=1}^n {k_r\choose 2}.$$

\begin{Theorem}[A bibasic Heine transformation formula; $A_n \to A_{m}$]\label{th:an-m-GB-Heine7} Let $0<|q^h|<1$,  $0<|q^t|<1$ and $0<\left| q^{ht}\right|<1$.  
 Further,  let
$|z/x_r|<1$, for $r=1,2,\dots, n$; and $|w/y_r|<1$, for $r=1,2,\dots, m$.  Then
\begin{align}
\multisum{k}{n} \Bigg( & \hvandermonde{x}{k}{n} 
\sqprod{n} \frac{\pqrfac{a_s\xover{x}}{k_r}{q^h} }{\pqrfac{q^h\xover{x}}{k_r}{q^h} } \notag \\
&\times \smallprod{m} \frac{\pqrfac{w/y_r}{h\sumk}{q^t}}{\pqrfac{b_rw/y_r}{h\sumk}{q^t}} \cdot
z^{\sumk} \hpowerq{k}{n} 
  \htqelementary{k}{h} \smallprod{n} x_r^{-k_r}
\Bigg) \notag \\
= \smallprod{m} & \frac{\pqrfac{w/y_r}{\infty}{q^t}}{\pqrfac{b_rw/y_r}{\infty}{q^t}} 
\smallprod{n} \frac{\pqrfac{a_r z /x_r}{\infty}{q^h}}{\pqrfac{z/x_r}{\infty}{q^h}}  \notag \\
\times \multisum{j}{m} & \Bigg( \htvandermonde{y}{j}{m}{t}
\sqprod{m} \frac{\pqrfac{{b_sy_r}/{y_s}}{j_r}{q^t} }{\pqrfac{q^t\xover{y}}{j_r}{q^t} } \notag \\
\times &  \smallprod{n} \frac{\pqrfac{z/x_r}{t\sumj}{q^h}}{\pqrfac{a_rz/x_r}{t\sumj}{q^h}} \cdot
w^{\sumj } \htpowerq{j}{m}{t} 
 \htqelementary{j}{t} \smallprod{m} y_r^{-j_r}
\Bigg). \label{an-m-GB-Heine7}
\end{align}
\end{Theorem}
\begin{proof} The proof is very similar in structure to the proof of \eqref{heine}. We begin with the left hand side, and write the factors
$$\smallprod{m} \frac{\pqrfac{w/y_r}{h\sumk}{q^t}}{\pqrfac{b_rw/y_r}{h\sumk}{q^t}}$$
as 
$$\smallprod{m} \frac{\pqrfac{w/y_r}{\infty}{q^t}}{\pqrfac{b_rw/y_r}{\infty}{q^t}}
\cdot\smallprod{m} \frac{\pqrfac{b_rwq^{th\sumk}/y_r}{\infty}{q^t}}{\pqrfac{wq^{th\sumk}/y_r}{\infty}{q^t}}.$$
We now use the $n=m$, $x_r\mapsto y_r$, $q\mapsto q^t$, $z\mapsto wq^{th\sumk}$, and 
$a_s\mapsto b_s$ case of \eqref{an-qbin2} to expand the second product. In this manner, we obtain a double sum, which can be represented symbolically as follows:
$$ \smallprod{m} \frac{\pqrfac{w/y_r}{\infty}{q^t}}{\pqrfac{b_rw/y_r}{\infty}{q^t}}
\multisum{k}{n} \ \ { } \multisum{j}{m} (\cdots) q^{\left(th\sumk\right)\sumj}.$$
On interchanging the sums, this can be written as
$$ \smallprod{m} \frac{\pqrfac{w/y_r}{\infty}{q^t}}{\pqrfac{b_rw/y_r}{\infty}{q^t}}
\multisum{j}{m} \ \ { } \multisum{k}{n} (\cdots) q^{\left(th\sumj\right)\sumk}.$$
Now the $q\mapsto q^h$, and  $z\mapsto zq^{th\sumj}$ case of \eqref{an-qbin2} and an elementary  
computation using \eqref{elementary1} immediately yields the right hand side.
\end{proof}
\begin{Remark}\label{rem:convergence} The absolute convergence of the series involved is shown using the multiple power series ratio test, and standard methods given in, for example, Milne \cite{Milne1997}. See also \cite{GB2017} where convergence requirements of terms such as $\pqrfac{w}{hk}{q^t}$ is explained. The condition $0<\left| q^{ht}\right|<1$ comes from such terms. The conditions $0<|q^h|<1$ and $0<|q^t|<1$ are required for the convergence of the infinite products appearing in the transformation formula.  From now on, we will refer to  these conditions as the {\em usual convergence conditions}. 
\end{Remark}

Observe that the  $A_n$ $q$-binomial theorem  \eqref{an-qbin2} is used twice in the proof of 
\eqref{an-m-GB-Heine7}.  A variation of this procedure is to use \eqref{an-qbin2} and another $A_n$ $q$-binomial theorem. In the next transformation formula, we use an $A_n$ $q$-binomial theorem obtained as a special case from an $A_n$ $_1\psi_1$ sum of Gustafson and Krattenthaler \cite[Eq.~(1.10)]{GK1996}, where we set $B=q$ and relabel some parameters. This result is: for $|z|<1$, 
\begin{align}
\smallprod{n} \frac{\qrfac{a zq^{r-1}}{\infty}}{\qrfac{zq^{r-1}}{\infty}}=
\multisum{k}{n} & \Bigg(  \vandermonde{x}{k}{n} 
 \notag \\
&\times 
\smallprod{n} \frac{\qrfac{a}{k_r} }{\qrfac{q}{k_r} } \cdot  z^{\sumk} \powerq{k}{n}  
\Bigg). 
\label{an-qbin3}
\end{align}

\begin{Theorem}[A bibasic Heine transformation formula; $A_n \to A_{m}$]
\label{th:an-m-GB-Heine8}
In addition to the usual convergence conditions, let $|w|<1$ and 
$|z/x_r|<1$, for $r=1, 2, \dots, n$. Then
\begin{align}
\multisum{k}{n} & \Bigg( \htvandermonde{x}{k}{n}{h}
\sqprod{n} \frac{\pqrfac{a_s\xover{x}}{k_r} {q^h}}{\pqrfac{q^h\xover{x}}{k_r} {q^h} } \notag \\
\times &  \smallprod{m} \frac{\pqrfac{wq^{t(r-1)}}{h\sumk}{q^t}}{\pqrfac{bwq^{t(r-1)}}{h\sumk}{q^t}}
\cdot z^{\sumk} \hpowerq{k}{n} 
 \htqelementary{k}{h} \smallprod{n} x_r^{-k_r}
\Bigg)
\notag \\
= \smallprod{m} & \frac{\pqrfac{wq^{t(r-1)}}{\infty}{q^t}}{\pqrfac{bwq^{t(r-1)}}{\infty}{q^t}} 
\smallprod{n} \frac{\pqrfac{a_r z/x_r}{\infty}{q^h}}{\pqrfac{z/x_r}{\infty}{q^h}}  \notag \\
\times \multisum{j}{m} \Bigg( & \htvandermonde{y}{j}{m}{t} 
\smallprod{m} \frac{\pqrfac{b}{j_r}{q^t} }{\pqrfac{q^t}{j_r}{q^t} } 
  \notag \\
&\times \smallprod{n} \frac{\pqrfac{z/x_r}{t\sumj}{q^h}}{\pqrfac{a_rz/x_r}{t\sumj}{q^h}} 
\cdot w^{\sumj} \htpowerq{j}{m}{t}  
\Bigg)
 \label{an-m-GB-Heine8}
\end{align}
\end{Theorem}
\begin{proof} The proof is analogous to the proof of Theorem~\ref{th:an-m-GB-Heine7}. The only difference is that we use \eqref{an-qbin3} on the left hand side to expand it as a double sum. The series on the right hand side converges when $|w|<1$.
\end{proof}

Observe that two  $A_n$ $q$-binomial theorems, namely \eqref{an-qbin2} and \eqref{an-qbin3} are used in the proof of \eqref{an-m-GB-Heine8}. Both are special cases of \eqref{an-m-GB-Heine8}. 
Take $h=1$ and $c=b$ to recover \eqref{an-qbin2}. 
Instead, set $t=1$, $a_r=1$  for $r=1, 2, \dots, n$ to obtain \eqref{an-qbin3}, after some re-labeling of parameters.


Next we use an $A_n$ $q$-binomial theorem with an extra parameter. For $|z|<1$:
\begin{align}\label{q-bin-GB-MS}
\frac{\pqrfac{a_1\cdots a_n z}{\infty}{q}}{\pqrfac{z}{\infty}{q}} =
\multisum{k}{n} & \Bigg(  \vandermonde{x}{k}{n}
\sqprod{n} \frac{\pqrfac{a_s\xover{x}}{k_r}{q} }{\pqrfac{q\xover{x}}{k_r}{q} } 
 \notag \\
&\times 
\smallprod n 
\frac{\pqrfac{cx_r/a_1\cdots a_n}{k_r}{q} \pqrfac{cx_r}{\sumk}{q} }
{\pqrfac{cx_r}{k_r}{q} \pqrfac{cx_r/a_r}{\sumk}{q}} 
\cdot z^{\sumk} \powerq{k}{n} 
\Bigg).
\end{align}
This follows easily from a fundamental lemma in the author's work with Schlosser \cite[Th.\ 11.2]{BS2018a}, where we take $p=0$. 
Observe that \eqref{q-bin-GB-MS} reduces to \eqref{q-bin},  the classical $q$-binomial theorem when $n=1$, but has an extra parameter $c$ which appears when $n>1$. When $c=0$, it reduces to an $A_n$ $q$-binomial theorem of Milne \cite[Th.\ 5.38]{Milne1997}. However, on examining the proof in \cite{BS2018a} carefully, we see that \eqref{q-bin-GB-MS} follows from an $A_n$ generalization of Jackson's sum due to Milne \cite[Th.\ 6.17]{Milne1988}.  More precisely, if we expand the left hand side using the classical $q$-binomial theorem, and compare coefficients of $z^N$, we obtain a result which is an $A_{n-1}$ generalization of a Jackson summation theorem (written in the form \cite[Lemma 3.26]{BM1997} or  \cite[Th.\ 5.1]{HR2004} with $p=0$).

\begin{Theorem}[A bibasic Heine transformation formula; $A_n \to A_{m}$]
In addition to the usual convergence conditions, let $|z|<1$ and $|w|<1$. Then,
\begin{align}
\multisum{k}{n} \Bigg( & \htvandermonde{x}{k}{n}{h}
\sqprod{n} \frac{\pqrfac{a_s\xover{x}}{k_r}{q^h} }{\pqrfac{q^h\xover{x}}{k_r}{q^h} } 
\cdot z^{\sumk} \hpowerq{k}{n} 
 \notag \\
&\times 
\smallprod n 
\frac{\pqrfac{cx_r/a_1\cdots a_n}{k_r}{q^h} \pqrfac{cx_r}{\sumk}{q^h} }
{\pqrfac{cx_r}{k_r}{q^h} \pqrfac{cx_r/a_r}{\sumk}{q^h}} 
\cdot \frac{\pqrfac{w}{h\sumk}{q^t}}{\pqrfac{b_1\cdots b_mw}{h\sumk}{q^t}}
\Bigg) \notag \\
&=  \frac{\pqrfac{w}{\infty}{q^t}}{\pqrfac{b_1\cdots b_m w }{\infty}{q^t}}
\; \frac{\pqrfac{a_1\cdots a_n z}{\infty}{q^h}}{\pqrfac{z}{\infty}{q^h}}\notag \\
\times \multisum{j}{m} & \Bigg( \htvandermonde{y}{j}{m}{t}
\sqprod{m} \frac{\pqrfac{b_sy_r/y_s }{j_r}{q^t}}{\pqrfac{q^t\xover{y}}{j_r}{q^t} } 
\cdot w^{\sumj} \htpowerq{j}{m}{t} 
\cr
&\times \smallprod m 
\frac{\pqrfac{dy_r/b_1\cdots b_m}{j_r}{q^t} \pqrfac{dy_r}{\sumj}{q^t} }
{\pqrfac{dy_r}{j_r}{q^t} \pqrfac{dy_r/b_r}{\sumj}{q^t}} 
\cdot \frac{\pqrfac{z}{t\sumj}{q^h}}{\pqrfac{a_1\cdots a_nz}{t\sumj}{q^h}}
\Bigg). \label{an-m-GB-Heine1}
\end{align}
\end{Theorem}
\begin{proof} The proof involves the use of \eqref{q-bin-GB-MS} twice and is very similar to the proof of Theorem~\ref{th:an-m-GB-Heine7}. We leave the details to the reader.  
\end{proof}
Observe that when $n=1=m$ in \eqref{an-m-GB-Heine1}, then the parameters $c$ and $d$ disappear, and we obtain \eqref{GB-sym-Heine}. 

Finally, we note one more formula, obtained by combining the $c=0$ case of \eqref{q-bin-GB-MS} (a summation formula of Milne \cite[Th.\ 5.38]{Milne1997}) with \eqref{an-qbin2}. These $A_n$ $q$-binomial theorems were among the first ones in this theory, and therefore we felt it appropriate to end with a formula combining the two. 
\begin{Theorem}[A bibasic Heine transformation formula; $A_n \to A_{m}$]\label{th:an-m-GB-Heine2}
In addition to the usual convergence conditions, 
let $|z|<1$ and 
$|w/y_r|<1$, for $r=1, 2, \dots, m$. Then
\begin{align}
\multisum{k}{n} \Bigg( & \htvandermonde{x}{k}{n}{h} 
\sqprod{n} \frac{\pqrfac{a_s\xover{x}}{k_r}{q^h} }{\pqrfac{q^h\xover{x}}{k_r}{q^h} } \notag \\
&\times \smallprod{m} \frac{\pqrfac{w/y_r}{h\sumk}{q^t}}{\pqrfac{b_rw/y_r}{h\sumk}{q^t}}
\cdot z^{\sumk} \hpowerq{k}{n} \Bigg) \notag \\
= \smallprod{m} & \left[ \frac{\pqrfac{w/y_r}{\infty}{q^t}}{\pqrfac{b_rw/y_r}{\infty}{q^t}}\right] 
\; \frac{\pqrfac{a_1a_2\cdots a_n z}{\infty}{q^h}}{\pqrfac{z}{\infty}{q^h}}\notag \\
\times \multisum{j}{m} & \Bigg( \htvandermonde{y}{j}{m}{t}
\sqprod{m} \frac{\pqrfac{b_sy_r/y_s}{j_r}{q^t} }{\pqrfac{q^t\xover{y}}{j_r}{q^t} } \notag \\
& \times  \frac{\pqrfac{z}{t\sumj}{q^h}}{\pqrfac{a_1a_2\cdots a_nz}{t\sumj}{q^h}}
w^{\sumj} \htpowerq{j}{m}{t} 
 \htqelementary{j}{t} \smallprod{m} y_r^{-j_r}
\Bigg). \label{an-m-GB-Heine2}
\end{align}
\end{Theorem}
\begin{proof} The proof involves the use of \eqref{q-bin-GB-MS} (with $c=0$) and \eqref{an-qbin2} and is very similar to the proof of Theorem~\ref{th:an-m-GB-Heine7}.
\end{proof}

We will not provide any further $A_n$ bibasic Heine transformations. But it is perhaps useful to note the various types of products that appear on the product sides of other $q$-binomial theorems in the literature. The following are the product sides of $A_n$ $q$-binomial theorems due to Milne \cite[Th.\ 5.40]{Milne1997}, Milne \cite[Th.\ 5.42]{Milne1997}, and Kajihara \cite[Eq.\ (3.6)]{Ka2012}, respectively:
\begin{equation*}
\frac{\smallprod{n}  \qrfac{azx_r}{\infty}
}{ \qrfac{z}{\infty}},
\frac{ \qrfac{az}{\infty} }{\qrfac{z}{\infty}}, 
\text{ and }
\smallprod{n} \frac{\qrfac{a_r z /x_r}{\infty}}{\qrfac{a_rz/a_1a_2\cdots a_n x_r}{\infty}} .
\end{equation*}
For all the above, one can write down multiple series bibasic Heine transformations with two sets of factors chosen from the 6 sets given by the product sides of the different $q$-binomial theorems.
Unfortunately, the $A_n$  $q$-binomial theorem due to the author and Schlosser \cite{BS1998} is not amenable to Heine's method. 


\section{Special cases: Ramanujan's $_2\phi_1$ transformations}\label{sec:ramanujan-2p1}

In this section, we find
multiple series extensions of many of Ramanujan's $_2\phi_1$ transformation formulas. These are special cases of the $A_n$ to $A_m$ bibasic Heine transformation formulas. In the classical case, this was done in \cite{GB2017}, which, in turn, is an addendum to Andrews and Berndt \cite[Ch.\ 1]{AB2009}. The identity central to the study of Ramanujan's transformations is given by 
\cite[Eq.~(3.1)]{GB2017} (mildly rewritten) 
\begin{align}\label{GB-1.4.1}
\frac{\pqrfac{aq^t}{\infty}{q^t} \pqrfac{cq^{h+1}}{\infty}{q^h} }{\pqrfac{-bq^{t+1}}{\infty}{q^t}\pqrfac{dq^h}{\infty}{q^h}} &
\sum_{j=0}^{\infty} \frac{\pqrfac{-bq/a}{j}{q^t}} {\pqrfac{q^t}{j}{q^t} }
\frac {\pqrfac{dq^h}{tj}{q^h}}{\pqrfac{cq^{h+1}}{tj}{q^h}} (aq^t)^j \cr
=&
\sum_{k=0}^{\infty} \frac{\pqrfac{cq/d}{k}{q^h}}{\pqrfac{q^h}{k}{q^h}}
\frac{ \pqrfac{aq^t}{hk}{q^t} }{\pqrfac{-bq^{t+1}}{hk}{q^t}} (dq^h)^k .
\end{align}
Again, $h$ and $t$ are complex numbers, and we have the usual convergence conditions, namely $|q^h|<1$, $|q^t|<1$ and $|q^{ht}|<1$.  Further, for the series to converge, we require $|aq^t|<1$ and $|dq^h|<1$.  
This is equivalent to \eqref{GB-sym-Heine}, as we can see by relabelling the parameters as follows: 
$a \mapsto cq/d$,  $b\mapsto -bq/a,$ $w\mapsto aq^t,$ and $z\mapsto  dq^h.$ When $h=2$ and $t=1$, this reduces to an equivalent form of Ramanujan's Entry 1.4.1 in \cite[\S 1.4]{AB2009}. 

In the special cases we consider, many of the products appearing in the sum simplify. 
Our first set of examples is obtained as special cases of Theorem~\ref{th:an-m-GB-Heine8}.
\begin{Corollary}[An extension of \eqref{GB-1.4.1}; $A_m\to A_n$] Aside from the usual convergence conditions, let $|aq^{tm}|<1$ and $|dq^{hn}|<1$. Then
\begin{align} 
\smallprod m \frac{\pqrfac{aq^{tmr}}{\infty}{q^{tm}}}{\pqrfac{-bq^{tmr+1}}{\infty}{q^{tm}}}
&
\cdot \frac{\pqrfac{cq^{h+1}}{\infty}{q^{h}}}{\pqrfac{dq^{h}}{\infty}{q^{h}}}\notag \\
\times \multisum{j}{m} & \Bigg( \vandespecialt{j}{m}{t}
\smallprod{m} \frac{\pqrfac{-bq/a}{j_r}{q^{tm}} }{\pqrfac{q^{tm}}{j_r}{q^{tm}} } \notag \\
& \times  \frac{\pqrfac{dq^{h}}{tmn\sumj}{q^{h}}}{\pqrfac{cq^{h+1}}{tmn\sumj}{q^{h}}}
\cdot (aq^{tm})^{\sumj} q^{tm \sum\limits_{r=1}^m (r-1)j_r} 
 \Bigg) \notag \\
=  \multisum{k}{n} \Bigg( & \vandespecialt{k}{n}{h} 
\smallprod{n} \frac{\pqrfac{cq^{1+h(r-n)}/d}{nk_r}{q^h} }{\pqrfac{q^{hr}}{nk_r}{q^{h}} } \notag \\
&\times \smallprod m 
\frac{\pqrfac{aq^{tmr}}{hn\sumk}{q^{tm}}}{\pqrfac{-bq^{1+tmr}}{hn\sumk}{q^{tm}}} \cr
&\times (dq^{hn})^{\sumk} q^{h(n-1)\sum\limits_{r=1}^n (r-1)k_r} 
\htqelementary{k}{hn}
\Bigg).
\label{an-m-GB-1.4.1-2}
\end{align}
\end{Corollary}
\begin{proof} We first obtain an equivalent formulation of \eqref{an-m-GB-Heine8} by taking
$a_r\mapsto c_rq/d$ (for $r=1, 2, \dots, n$), $z\mapsto dq^h$, $b\mapsto -bq/a$, and 
$w\mapsto aq^t$. In this manner, we obtain an extension of \eqref{GB-1.4.1}.  In the resulting identity, we take $c_1=c_2=\cdots = c_n=c$, $h\mapsto hn$, $t\mapsto tm$, and specialize $x_r$ and $y_r$ as follows
\begin{align*}
x_r &\mapsto q^{h(r-1)} \text{ (for $r=1, 2, \dots, n$)}\cr
y_r &\mapsto q^{t(r-1)} \text{ (for $r=1, 2, \dots, m$)}.
\end{align*}
With these special cases, several products appearing in the sum simplify. We note two examples of how different factors simplify.
\begin{align*}
\sqprod{n} \frac{\pqrfac{cx_rq/dx_s}{k_r} {q^h}}{\pqrfac{q^h\xover{x}}{k_r} {q^h} }
&\longrightarrow
\smallprod{n} \frac{\pqrfac{cq^{1+h(r-n)}/d}{nk_r}{q^h} }{\pqrfac{q^{hr}}{nk_r}{q^{h}} } 
\label{sqprod-simplify}
\\
 \smallprod{n} \pqrfac{dq^h/x_r}{t\sumj}{q^h} 
 &\longrightarrow
 \pqrfac{dq^{h}}{tmn\sumj}{q^{h}}.
 \end{align*}
Both of these simplifications are obtained using the elementary identity \cite[Eq.\ (I.27)]{GR90}
\begin{equation*}
\qrfac{a}{nk} = \pqrfac{a,aq,\dots, aq^{n-1}}{k}{q^n}.
\end{equation*}
It should now be clear how \eqref{an-m-GB-1.4.1-2} is obtained. 
\end{proof}

We can have many $A_m\to A_n$ generalizations of \eqref{GB-1.4.1} on the lines of \eqref{an-m-GB-1.4.1-2}. Special cases of these results lead to multiple series generalizations of many of Ramanujan's transformation formulas from \cite[Ch.\ 1]{AB2009}. We give a small sample of possibilities below, to illustrate the kinds of formulas that can be obtained. We have kept a convention of keeping the index $j$ for the $A_m$ series, and index $k$ for the $A_n$ series, to make it easy for the reader to see how the identities are obtained from the corresponding bibasic Heine transformations. 

Next we consider the case $a\to 0$ and $c=0$ in \eqref{an-m-GB-1.4.1-2}. In the resulting identity, set $b\mapsto -a/q$ and $d\mapsto b$, and then take $h=t=1$, and $a=b=1$. In this manner we obtain the following identity, which is a generalization of Ramanujan's \cite[Entry 1.4.10]{AB2009}. For $|q|<1$:
\begin{align}
\multisum{k}{n} \Bigg( & \vandespecial{k}{n} 
\smallprod{n} \frac{1 }{\qrfac{q^r}{nk_r}}
\smallprod{m} \frac{1}{\pqrfac{q^{mr}}{n\sumk}{q^m}}
 \notag \\
&\times 
q^{n\sumk} q^{(n-1) \sum\limits_{r=1}^n (r-1)k_r} 
\htqelementary{k}{n} 
\Bigg) \notag \\
=  & 
\frac{1}{\pqrfac{q}{\infty}{q}}
\smallprod m
\frac{1}{\pqrfac{q^{mr} }{\infty}{q^m}}
\notag \\
\times
 \multisum{j}{m} & \Bigg( \vandespecial{j}{m} \cdot
\frac{\qrfac{q}{mn\sumj} } {\smallprod{m} \pqrfac{q^m}{j_r}{q^m} } 
\notag \\
\times &  
(-1)^{\sumj} 
q^{m\sum\limits_{r=1}^m (r-1)j_r}
q^{m\sum\limits_{r=1}^m {j_r+1\choose 2}}
\Bigg).
\label{an-m-andrews6-1.4.10}
\end{align}
When $m=1$, this reduces to 
\begin{align}
\multisum{k}{n} \Bigg( & \vandespecial{k}{n}  
\smallprod{n} \frac{1 }{\qrfac{q^r}{nk_r}}
\cdot
 \frac{1}{\pqrfac{q}{n\sumk}{q}}
 \notag \\
&\times 
 q^{n\sumk} q^{(n-1) \sum\limits_{r=1}^n (r-1)k_r} 
\htqelementary{k}{n} 
\Bigg) \notag \\
=  &\frac{1}{(q;q)^2_{\infty}}
\sum_{j=0}^{\infty}
\frac{\qrfac{q}{nj} }{\qrfac{q}{j} } 
(-1)^{j} q^{\frac{j(j+1)}{2}} .
\label{an-m-andrews6-1.4.10b}
\end{align}
Compare these formulas with Ramanujan's own formula \cite[Entry 1.4.10]{AB2009}, obtained when $n=m=1$:
\begin{equation}\label{Entry1.4.10}
\sum_{k\geq 0} \frac{q^k}{(q;q)^2_k} =\frac{1}{(q;q)^2_{\infty}}
\sum_{j=0}^{\infty}
(-1)^{j} q^{\frac{j(j+1)}{2}} .
\end{equation}

We can obtain extensions of Entry 1.4.10 from other bibasic Heine transformations. For the sake of comparison, we present one obtained 
from Theorem \ref{th:an-m-GB-Heine2}. 
We take $a_r\mapsto c_rq/d$ (for $r=1, 2, \dots, n$), $z\mapsto d^nq^h$, $b_r\mapsto -b_rq/a$ (for $r=1,2,\dots, m$), and 
$w\mapsto aq^t$. In this manner, we obtain an extension of \eqref{GB-1.4.1}.  In the resulting identity, we take $b_1=b_2=\cdots = b_m=b$,  $c_1=c_2=\cdots = c_n=c$, $h\mapsto hn$, $t\mapsto tm$, and specialize $x_r$ and $y_r$ as follows
\begin{align*}
x_r &\mapsto q^{h(r-1)} \text{ (for $r=1, 2, \dots, n$)}\cr
y_r &\mapsto q^{t(r-1)} \text{ (for $r=1, 2, \dots, m$)}.
\end{align*}
Next we take $b=0$ and $d\to 0$, replace $c$ by $b/q$. Finally, we take $h=t=1$ and $a=b=1$ to obtain, for $|q|<1$:
\begin{align}
\multisum{j}{m} \Bigg( & \vandespecial{j}{m} 
\smallprod{m} \frac{1 }{\qrfac{q^r}{mj_r}} \cdot
\frac{1}{\pqrfac{q^{n}}{m\sumj}{q^n}}
 \notag \\
&\times 
q^{m\sumj} q^{(m-1) \sum\limits_{r=1}^n (r-1)j_r} 
\htqelementary{j}{m} 
\Bigg) \notag \\
&=   
\frac{1}{\pqrfac{q }{\infty}{q} \pqrfac{q^n}{\infty}{q^n}}
\notag \\
\times
 \multisum{k}{n} & \Bigg( \vandespecial{k}{n} \cdot
\frac{\qrfac{q}{mn\sumk} } {\smallprod{n} \qrfac{q^r}{nk_r} } 
\notag \\
& \times  
(-1)^{n\sumk} 
q^{2n\sum\limits_{r=1}^n (r-1)k_r -n(n-1)\sumk}
q^{\sum\limits_{r=1}^n {nk_r+1\choose 2}}
\Bigg).
\label{an-m-andrews6-1.4.10-c}
\end{align}
When $n=1$, this reduces to \eqref{an-m-andrews6-1.4.10b}. But if $m=1$, we obtain 

\begin{align}
\sum_{j=0}^{\infty} &
\frac{ q^j}{\qrfac{q}{j} \pqrfac{q^n}{j}{q^n}} 
=  
\frac{1}{(q;q)_{\infty}(q^n;q^n)_{\infty}} 
\multisum{k}{n} \Bigg(  \vandespecial{k}{n} \cr 
&\times   \frac{\pqrfac{q}{n\sumk}{q} }
{\smallprod{n} \qrfac{q^r}{nk_r}}  
(-1)^{n\sumk} 
q^{2n\sum\limits_{r=1}^n (r-1)k_r -n(n-1)\sumk}
q^{\sum\limits_{r=1}^n {nk_r+1\choose 2}}
\Bigg).
\end{align}

Next we consider another fruitful limiting case of \eqref{GB-1.4.1}, where we take $a\to 0$ and $d\to 0$. This limiting case yields many identities of Ramanujan. 

%
We cannot take the limit as $d\to 0$ in \eqref{an-m-GB-1.4.1-2} unless $n=1$. So we take  $n=1$, and  $a\to 0$ and $d\to0$ and then take $c\mapsto -a/q$ and $b\mapsto b/q$. In this manner, assuming the usual convergence conditions, we obtain:
\begin{align}
\pqrfac{-aq^{h}}{\infty}{q^{h}} 
\multisum{j}{m} & \Bigg( \vandespecialt{j}{m}{t}
\smallprod{m} \frac{1}{\pqrfac{q^{tm}}{j_r}{q^{tm}} } 
\notag \\
& \times
  \frac{b^{\sumj} }{\pqrfac{-aq^{h}}{tm\sumj}{q^{h}}}
q^{tm\sum\limits_{r=1}^m (r-1)j_r}
q^{tm\sum\limits_{r=1}^m {j_r+1\choose 2}}
\Bigg) \notag \\
= 
\smallprod m \pqrfac{-bq^{tmr}}{\infty}{q^{tm}} 
&\sum_{k=0}^{\infty}
\frac{a^kq^{h{k+1\choose 2}}}{\pqrfac{q^{h}}{k}{q^{h}} 
\smallprod m \pqrfac{-bq^{tmr}}{hk}{q^{tm}}}
.
\label{an-m-GB-2016-eq26-a2}
\end{align}
Compare this identity with the identity \cite[Eq.\ (3.17)]{GB2017}:
\begin{align}\label{GB-1.4.12}
\pqrfac{-aq^h}{\infty}{q^h} 
\sum_{j=0}^{\infty} 
\frac {b^j q^{t{j+1\choose 2}} }{\pqrfac{q^t}{j}{q^t} \pqrfac{-aq^h}{tj}{q^h}} 
=
\pqrfac{-bq^t}{\infty}{q^t} \sum_{k=0}^{\infty} 
\frac{ a^k q^{h{k+1\choose 2}} }{\pqrfac{q^h}{k}{q^h} \pqrfac{-bq^t}{hk}{q^t}} .
\end{align}

We can take $h=t$ in \eqref{an-m-GB-2016-eq26-a2} and let $q\mapsto q^{1/t}$, to obtain an extension of \eqref{1.4.17}. For $|q|<1$, $|q^{t}|<1$: 
\begin{align}
\pqrfac{-aq}{\infty}{q} 
\multisum{j}{m} & \Bigg( \vandespecial{j}{m}
\smallprod{m} \frac{1}{\pqrfac{q^{m}}{j_r}{q^{m}} } 
\notag \\
& \times
  \frac{b^{\sumj} }{\pqrfac{-aq}{tm\sumj}{q}}
q^{m\sum\limits_{r=1}^m (r-1)j_r}
q^{m\sum\limits_{r=1}^m {j_r+1\choose 2}}
\Bigg) \notag \\
= 
\smallprod m \pqrfac{-bq^{mr}}{\infty}{q^{m}} 
&\sum_{k=0}^{\infty}
\frac{a^kq^{{k+1\choose 2}}}{\pqrfac{q}{k}{q} 
\smallprod m \pqrfac{-bq^{mr}}{tk}{q^{m}}}
.
\label{an-m-GB-2016-eq26-a3}
\end{align}

Next we obtain analogous results from the  $c=0=d$ case of \eqref{an-m-GB-Heine1}. First, we obtain a result analogous to \eqref{an-m-GB-1.4.1-2}. 
We take
$a_r\mapsto c_rq/d$ (for $r=1, 2, \dots, n$), $z\mapsto d^nq^h$, 
$b_r\mapsto -b_rq/a$ (for $r=1, 2, \dots, m$), and 
$w\mapsto a^mq^t$. In this manner, we obtain an extension of \eqref{GB-1.4.1}.  In the resulting identity, we take  $b_1=b_2=\cdots = b_m=b$; $c_1=c_2=\cdots = c_n=c$, $h\mapsto hn$, $t\mapsto tm$, and specialize $x_r$ and $y_r$ as follows
\begin{align*}
x_r &\mapsto q^{h(r-1)} \text{ (for $r=1, 2, \dots, n$)}\cr
y_r &\mapsto q^{t(r-1)} \text{ (for $r=1, 2, \dots, m$)}.
\end{align*}
Next, we take the limits as $a\to 0$ and $d\to 0$, replace $c$ by $-a/q$ and $b$ by $b/q$ and obtain the following identity. 
\begin{align}
\pqrfac{(-aq^h)^n}{\infty}{q^{hn}} 
\multisum{j}{m} & \Bigg( \vandespecialt{j}{m}{t}
\notag \\
& \times
\smallprod{m} \frac{1}{\pqrfac{q^{tr}}{mj_r}{q^t} } \cdot
  \frac{b^{m\sumj} }{\pqrfac{(-aq^h)^n}{tm\sumj}{q^{hn}}}
  \notag \\
& \times
q^{2tm\sum\limits_{r=1}^m (r-1)j_r -tm(m-1)\sumj}
q^{t\sum\limits_{r=1}^m {mj_r+1\choose 2}}
\Bigg) \notag \\
= 
\pqrfac{(-bq^t)^m}{\infty}{q^{tm}} \multisum{k}{n} \Bigg( & \vandespecialt{k}{n}{h} 
\cr
&\times \smallprod{n} \frac{1}{\pqrfac{q^{hr}}{nk_r}{q^{h}} }\cdot
 \frac{a^{n\sumk} }{\pqrfac{(-bq^t)^m}{hn\sumk}{q^{tm}}}
 \cr
&\times
q^{2hn\sum\limits_{r=1}^n (r-1)k_r -hn(n-1)\sumk}
q^{h\sum\limits_{r=1}^n {nk_r+1\choose 2}}
\Bigg).
\label{an-m-GB2016-eq26-b}
\end{align}
Again, we require the usual convergence conditions. Note that the summands on either side have powers of $q$ that are quadratic in the indices of summation. This ensures that we need no further conditions on $a$ and $b$ for the series to converge. 

In the classical case, i.e., when $n=m=1$, this reduces to \eqref{GB-1.4.12}. If we take $h=t$ in \eqref{an-m-GB2016-eq26-b} and replace $q$ by $q^{1/t}$, we obtain \eqref{an-m-GB-1.4.17b}, an identity we highlighted in the introduction. 
Further, take $t=1$, $a=-1$, $b=1$, and $n=m$ in 
\eqref{an-m-GB-1.4.17b} to obtain
\begin{align}
\multisum{j}{m} & \Bigg( \vandespecial{j}{m}
\smallprod{m} \frac{1}{\pqrfac{q^{r}}{mj_r}{q} } \cdot
  \frac{1}{\pqrfac{q^{m}}{m\sumj}{q^{m}}}
  \notag \\
& \times
q^{2m\sum\limits_{r=1}^m (r-1)j_r - m(m-1)\sumj}
q^{\sum\limits_{r=1}^m {mj_r+1\choose 2}}
\Bigg) \notag \\
&= 
\frac{\pqrfac{(-q)^{m}}{\infty}{q^{m}}} 
{\pqrfac{q^{m}}{\infty}{q^{m}} } \cr
&\times \multisum{k}{m} \Bigg(  \vandespecial{k}{m}
 \smallprod{m} \frac{1}{\pqrfac{q^{r}}{mk_r}{q} }\cdot
 \frac{(-1)^{m\sumk} }{\pqrfac{(-q)^{m}}{m\sumk}{q^{m}}}
 \cr
&\times
q^{2m\sum\limits_{r=1}^m (r-1)k_r -m(m-1)\sumk}
q^{\sum\limits_{r=1}^m {mk_r+1\choose 2}}
\Bigg).
\label{an-m-GB-1.4.9a}
\end{align}
Compare this with Ramanujan's \cite[Entry 1.4.9]{AB2009}
\begin{equation}\label{1.4.9}
 \sum_{j=0}^\infty \frac{q^{j+1\choose 2}}{\left(q;q\right)^2_j} 
 =\frac{\qrfac{-q}{\infty}}{\left(q;q\right)_{\infty}}
\sum_{k=0}^{\infty}
 \frac{(-1)^{k} q^{k+1\choose 2}}{\pqrfac{q}{k}{q}\pqrfac{-q}{k}{q}},
\end{equation}
where we have used $\pqrfac{q^2}{k}{q^2} =\pqrfac{q}{k}{q}\pqrfac{-q}{k}{q}$ to rephrase
\cite[Entry 1.4.9]{AB2009}. 
Consider once more \eqref{an-m-GB-1.4.17b}, with
$t=1$, $a=-1$, $b=1$, and $n=1$. In this case we obtain the following extension of 
\eqref{1.4.9}.
\begin{align}
\multisum{j}{m} & \Bigg( \vandespecial{j}{m}
\smallprod{m} \frac{1}{\pqrfac{q^{r}}{mj_r}{q} } \cdot
  \frac{1}{\pqrfac{q}{m\sumj}{q}}
  \notag \\
& \times
q^{2m\sum\limits_{r=1}^m (r-1)j_r - m(m-1)\sumj}
q^{\sum\limits_{r=1}^m {mj_r+1\choose 2}}
\Bigg) \notag \\
&= 
\frac{\pqrfac{(-q)^{m}}{\infty}{q^{m}}} 
{\pqrfac{q}{\infty}{q} }
\sum_{k=0}^{\infty}
 \frac{(-1)^{k} q^{k+1\choose 2}}{\pqrfac{q}{k}{q}\pqrfac{(-q)^m}{k}{q^m}}.
\label{an-m-GB-1.4.9b}
\end{align}

The above should serve to demonstrate the kind of multiple series extensions of Ramanujan $_2\phi_1$ transformations possible. Rewriting the bibasic Heine transformations in a format extending identity \eqref{GB-1.4.1} is the key idea to obtain such special cases. From then on, the calculations shown in \cite{GB2017} can be used to find such transformations.  As we have seen, some special cases simplify the products appearing in the identities.  Clearly, many, many generalizations of Ramanujan's identities can be obtained in this manner.

\section{A variation: Using the $q$-Euler transformation formula}\label{sec:q-euler}

In Heine's method, 
instead of the $q$-Binomial Theorem, we can use
the second iterate of Heine's transformation  (a $q$-analogue of Euler's transformation of hypergeometric functions) \cite[Eq.\ (1.4.3)]{GR90}. Let $|q|<1$, $|z|<1$ and $|abz/c|<1$. Then
\begin{equation}\label{q-euler}
\sum_{k=0}^{\infty} \frac{\pqrfac{a, b}{k}{q} }{\pqrfac{q, c}{k}{q}} z^k 
=
\frac{\pqrfac{abz/c}{\infty}{q} }{\pqrfac{z}{\infty}{q}}
\sum_{j=0}^{\infty} \frac{\qrfac{c/a, c/b}{j}} {\qrfac{q, c}{j} } \left(abz/c\right)^j .
\end{equation}
Note that when $c=b$, this reduces to the $q$-binomial theorem. Using the $q$-analogue of Euler's transformation formula, we obtain the following transformation formula of double sums.
\begin{Theorem} Aside from the usual convergence conditions (see Remark~\ref{rem:convergence}), let 
$|z|<1$, $|w|<1$,  $|abzq^{ht}/c|<1$ and $|dewq^{ht}/f|<1$. Then
\begin{align}\label{GB-bibasic-Euler}
\sum_{k\geq 0} & \frac{\pqrfac{a,b}{k}{q^h}  }{\pqrfac{q^h, c}{k}{q^h}}
\frac{ \pqrfac{w}{hk}{q^t} }{\pqrfac{dew/f}{hk}{q^t}} z^k
\sum_{\widetilde{k}\geq 0}  
\frac{\pqrfac{f/d, f/e}{\widetilde{k}}{q^t} }{\pqrfac{q^t,f}{\widetilde{k}}{q^t}
} 
\left( dew q^{htk}/f \right)^{\widetilde{k}} 
\cr
& = \frac{\pqrfac{w}{\infty}{q^t}}{\pqrfac{dew/f}{\infty}{q^t}}
\frac{\pqrfac{abz/c}{\infty}{q^h} }{\pqrfac{z}{\infty}{q^h}}
\cr
&\times
\sum_{j\geq 0}  \frac{\pqrfac{d, e}{j}{q^t}  }{\pqrfac{q^t, f}{j}{q^t}}
\frac{ \pqrfac{z}{tj}{q^h} }{\pqrfac{abz/c}{tj}{q^h}} w^j
\sum_{\widetilde{j}\geq 0}  
\frac{\pqrfac{c/a, c/b}{\widetilde{j}}{q^h} }{\pqrfac{q^h, c}{\widetilde{j}}{q^h}} 
\left( abzq^{htj}/c  \right)^{\widetilde{j}} .
\end{align}
\end{Theorem}
\begin{proof} The proof is a straightforward extension of the proof of \eqref{heine}. We use \eqref{q-euler} with two sets of variables $(a, b, c, z)$ with base $q^h$ and 
$(d, e, f, w)$ with base $q^t$. The rest of the details are quite similar. We give more details when extending to multiple series, below. 
\end{proof}
Note that when $c=b$ and $f=e$, then \eqref{GB-bibasic-Euler} reduces to \eqref{GB-sym-Heine}.

There are many $A_n$ extensions of \eqref{q-euler}. They follow from multiple series extensions of Bailey's $_{10}\phi_9$ transformation formulas. See \cite[Th.\ 5.15]{BS1998} for an example of this calculation. Many of them can be used to generalize 
\eqref{GB-bibasic-Euler}. As an example, we use a result of Kajihara \cite[Th.\ 1.1]{Ka2004}.  We chose this particular transformation formula, because it was one of the first results which transformed an $A_n$ series into an $A_m$ series---and therefore, its existence was an important motivation of our work.  Kajihara's transformation formula can be stated as follows. For $|z|<1$:
\begin{align}
\multisum{k}{n} \Bigg( & \vandermonde{x}{k}{n} 
\sqprod{n} \frac{\pqrfac{a_s\xover{x}}{k_r}{q} }{\pqrfac{q\xover{x}}{k_r}{q} }
\cr
&\times
 \smallprod{n}\prod_{s=1}^m 
\frac{\pqrfac{b_s{x_r}{y_s}}{k_r}{q}}
{\pqrfac{c{x_r}{y_s}}{k_r}{q}}z^{\sumk} \powerq{k}{n} 
\Bigg)
\notag\\
&= 
\frac{\pqrfac{a_1\cdots a_n b_1\cdots b_m z/c^m}{\infty}{q}}{\pqrfac{z}{\infty}{q}}\notag \\
\times \multisum{j}{m} \Bigg( & \vandermonde{y}{j}{m} 
\sqprod{m} \frac{\pqrfac{{cy_r}/{b_sy_s}}{j_r}{q} }{\pqrfac{q\xover{y}}{j_r}{q} }
 \smallprod{m}\prod_{s=1}^n 
\frac{\pqrfac{cx_sy_r/a_s}{j_r}{q}}
{\pqrfac{c{x_s}{y_r}}{j_r}{q}} 
 \notag \\
&\times 
 \left( {a_1\cdots a_n b_1\cdots b_m z}/{c^m}\right)^{\sumj} 
\powerq{j}{m}
\Bigg). \label{an-m-kajihara}
\end{align}

Using Heine's method on this transformation formula, we obtain the following transformation formula of double multiple-sums.
\begin{Theorem}[An extension of \eqref{GB-bibasic-Euler}; 
$A_{n}\textnormal{-}A_\nu\to A_{m}\textnormal{-}A_\mu$] 
Let 
$|z|<1$, $|w|<1$,
$ \left| a_1\cdots a_n b_1\cdots b_\mu zq^{th}/ c^\mu\right|<1$, and
  $\left| d_1\cdots d_m e_1\cdots e_\nu wq^{th}/f^\nu\right|<1$. In addition, suppose the usual convergence conditions apply. 
  Then, we have
\begin{align}
\multisum{k}{n} \Bigg( & \hvandermonde{x}{k}{n} 
\sqprod{n} \frac{\pqrfac{a_s\xover{x}}{k_r}{q^h} }{\pqrfac{q^h\xover{x}}{k_r}{q^h} } 
\smallprod{n}\prod_{s=1}^\mu 
\frac{\pqrfac{b_sx_rX_s}{k_r}{q^h}}
{\pqrfac{cx_rX_s}{k_r}{q^h}}
\notag \\
&\times 
\frac{\pqrfac{w}{h\sumk}{q^t}}
{\pqrfac{d_1\cdots d_m e_1\cdots e_\nu w/{f^\nu}}{h\sumk }{q^t}} 
\cdot z^{\sumk} \hpowerq{k}{n} 
\notag\\
\times \multisum{\widetilde{k}}{\nu} \Bigg( & \htvandermonde{Y}{\widetilde{k}}{\nu}{t} 
\sqprod{\nu} \frac{\pqrfac{fY_r/e_s Y_s}{\widetilde{k}_r}{q^t} }
{\pqrfac{q^tY_r/Y_s}{\widetilde{k}_r}{q^t} } 
\smallprod{\nu}\prod_{s=1}^m 
\frac{\pqrfac{fy_sY_r/d_s}{\widetilde{k}_r}{q^t}}
{\pqrfac{fy_sY_r}{\widetilde{k}_r}{q^t}} \notag\\
&\times \left( d_1\cdots d_m e_1\cdots e_\nu wq^{th\sumk}/f^\nu\right)^{\sumkcurl}  \htpowerq{\widetilde{k}}{\nu}{t}
 \Bigg)\Bigg) \notag \\
= & \frac{\pqrfac{w}{\infty}{q^t}\pqrfac{a_1\cdots a_n b_1\cdots b_\mu z/c^\mu}{\infty}{q^h}}
{\pqrfac{d_1\cdots d_m e_1\cdots e_\nu w/f^\nu}{\infty}{q^t} \pqrfac{z}{\infty}{q^h}}
\notag \\
\times\multisum{j}{m} \Bigg( & \htvandermonde{y}{j}{m}{t} 
\sqprod{m} \frac{\pqrfac{d_s\xover{y}}{j_r}{q^t} }{\pqrfac{q^t\xover{y}}{j_r}{q^t} } 
\smallprod{m}\prod_{s=1}^\nu 
\frac{\pqrfac{e_sy_rY_s}{j_r}{q^t}}
{\pqrfac{fy_rY_s}{j_r}{q^t}}
\notag \\
&\times 
 \frac{\pqrfac{z}{t\sumj}{q^h}}
 {\pqrfac{a_1\cdots a_n b_1\cdots b_\mu z/{c^\mu}}{t\sumj}{q^h}} 
 \cdot w^{\sumj} \htpowerq{j}{m}{t}
  \notag\\
\times \multisum{\widetilde{j}}{\mu} \Bigg( & \hvandermonde{X}{\widetilde{j}}{\mu} 
\sqprod{\mu} \frac{\pqrfac{cX_r/b_sX_s}{\widetilde{j}_r}{q^h} }
{\pqrfac{q^h\xover{X}}{\widetilde{j}_r}{q^h} }
\smallprod{\mu}\prod_{s=1}^n 
\frac{\pqrfac{cx_sX_r/a_s}{\widetilde{j}_r}{q^h}}
{\pqrfac{cx_sX_r}{\widetilde{j}_r}{q^h}} 
 \notag \\
&\times 
\left( a_1\cdots a_n b_1\cdots b_\mu zq^{th\sumj }/ c^\mu\right)^{\sumjcurl} \hpowerq{\widetilde{j}}{\mu}
\Bigg)\Bigg). \label{an-m-GB-kajihara1 }
\end{align}
\end{Theorem}
\begin{proof}
The proof is similar to the proof to that of Theorem~\ref{th:an-m-GB-Heine7}. We indicate the steps symbolically. First represent \eqref{an-m-kajihara} symbolically as:
$$\sum_{k} L(a,b,c; x; z; q; n; k) = P(a,b,c; z; q) \sum_{j}R(a,b,c; y; z; q; m; j).$$
Of course, $m$ shows up on the left hand side too, and $n$ appears on the right hand side, but we suppress it  for clarity. (Similarly, we suppress some other symbols too.) The symbols $L, P, R$ are code for Left, Products, Right, respectively.  We will use two copies of this result. Let us represent the summands of the corresponding left hand sides by $L_1$ and $L_2$, and similarly have corresponding $P_1, P_2, R_1, R_2$. We have
\begin{align*}
L_1 = L(a,b,c; x; zq^{ht\sumj}; q^h; n; k) & \text{ and } R_1 = R(a,b,c; X; zq^{ht\sumj}; q^h; \mu; \widetilde{j}); \cr
L_2 = L(d,e,f; y; wq^{ht\sumk}; q^t; m; j) & \text{ and } 
R_2 = R(d,e,f; Y; wq^{ht\sumk}; q^t; \nu; \widetilde{k}).
\end{align*}
Let $P_1$ and $P_2$ be the corresponding products. 

The first copy of  \eqref{an-m-kajihara}  is a $A_n\to A_\mu$ transformation, where we use the indices of summation
$k$ and $\widetilde{j}$, and the second copy is a $A_m\to A_\nu$ transformation, with indices
of summation
$j$ and $\widetilde{k}$.

We now see that the left hand side can be written as
\begin{align*}
 \frac{\pqrfac{w}{\infty}{q^t}}{\pqrfac{d_1\cdots d_m e_1\cdots e_\nu w/f^\nu}{\infty}{q^t}} 
 & \sum_k (\cdots) \times P_2 \sum_{\widetilde{k}} R_2 \cr
 &= 
\frac{\pqrfac{w}{\infty}{q^t}}{\pqrfac{d_1\cdots d_m e_1\cdots e_\nu w/f^\nu}{\infty}{q^t}} 
 \sum_k (\cdots) \times \sum_{j} L_2 \cr
 &=  
 \frac{\pqrfac{w}{\infty}{q^t}}{\pqrfac{d_1\cdots d_m e_1\cdots e_\nu w/f^\nu}{\infty}{q^t}} 
 \sum_j (\cdots) \times \sum_{k} L_1 \cr
 &=
 \frac{\pqrfac{w}{\infty}{q^t}}{\pqrfac{d_1\cdots d_m e_1\cdots e_\nu w/f^\nu}{\infty}{q^t}} 
 \sum_j (\cdots) \times P_1 \sum_{\widetilde{j}} R_1 \cr
=
 \frac{\pqrfac{w}{\infty}{q^t}}{\pqrfac{d_1\cdots d_m e_1\cdots e_\nu w/f^\nu}{\infty}{q^t}} 
& \frac{\pqrfac{a_1\cdots a_n b_1\cdots b_\mu z/c^\mu}{\infty}{q^h}}{\pqrfac{z}{\infty}{q^h}}
 \sum_j (\cdots) \sum_{\widetilde{j}} R_1.
\end{align*}
We recognize the right hand side at the last step. This completes  the proof.
 \end{proof}

\section{Iterating Heine's method: Andrews' $q$-Lauricella transformation}\label{sec:qlauricella}
We now present another variation, where Heine's method is iterated, to  obtain a transformation formula that transforms a multiple sum to a single sum. This was first done by Andrews~\cite[Eq.~(4.1)]{An1972} when he obtained his transformation formula for the $q$-Lauricella function. Andrews' formula was generalized to a multibasic formula by Agarwal, Jain and Choi \cite{AJC2017}, and independently,  the author \cite{GB2017}, which is what we extend to $A_n$ series in this section. More precisely, we will indicate how such generalizations can be achieved, and show a couple of  samples. A result of this type transforms multiple multiple-sums, into a single multiple sum. 

We use the symbol for the dot product
$$ {{\h}\cdot{\k}} = h_1k_1+h_2k_2+\cdots + h_pk_p.$$
With this notation, we have the formula given by \cite[Th.\ 2.2]{AJC2017} and \cite[Eq.\ (5.1)]{GB2017}:
\begin{align}\label{GB-qlauricella}
\multisum{k}{p} \ &
\smallprod{p} \frac{\pqrfac{a_r}{k_r}{q^{h_r}}}{\pqrfac{q^{h_r}}{k_r}{q^{h_r}}}
\frac{\pqrfac{w}{ {\h}\cdot{\k}}{q^t}}
{\pqrfac{bw}{{\h}\cdot{\k}}{q^t}}
 \smallprod{p} z_r^{k_r} \cr
& =
\frac{\pqrfac{w}{\infty}{q^t}}{\pqrfac{bw}{\infty}{q^t}}
\smallprod{p} \frac{\pqrfac{a_rz_r}{\infty}{q^{h_r}}}{\pqrfac{z_r}{\infty}{q^{h_r}}}
\sum_{j=0}^{\infty} \frac{\pqrfac{b}{j}{q^t}} {\pqrfac{q^t}{j} {q^t}}
\smallprod{p} \frac{\pqrfac{z_r}{tj}{q^{h_r}}}{\pqrfac{a_rz_r}{tj}{q^{h_r}}}
w^j.
\end{align}
The relevant convergence conditions are direct extensions of the conditions in \eqref{GB-sym-Heine} to which the identity reduces when $p=1$. When $a_r\mapsto b_r$, $b\mapsto c/a$, $w\mapsto a$, and  $h_1=h_2=\cdots h_p=1=t$, then \eqref{GB-qlauricella} reduces to a transformation of Andrews for $q$-Lauricella functions 
\cite[Eq.~(4.1)]{An1972}. 

Note that the left hand side of \eqref{GB-qlauricella} is a $p$-fold sum and the right hand side is a single sum. The identity has $p+1$ bases, namely $q^t$ and $q^{h_1}, \dots, q^{h_p}$. Note further that if $p=1$, then it reduces to \eqref{GB-sym-Heine}.

Before proceeding with describing its $A_n$ extension, we recall the key idea of the proof from \cite{GB2017}. The idea is to iterate Heine's method $p$ times. 
The first step of the proof is as follows.  
\begin{align*}
\multisum{k}{p} &
\smallprod{p} \frac{\pqrfac{a_r}{k_r}{q^{h_r}}}{\pqrfac{q^{h_r}}{k_r}{q^{h_r}}} 
\frac{\pqrfac{w}{ {\h}\cdot{\k}}{q^t}}
{\pqrfac{bw}{{\h}\cdot{\k}}{q^t}}
 \smallprod{p} z_r^{k_r} \cr
&=
\frac{\pqrfac{w}{\infty}{q^t}}{\pqrfac{bw}{\infty}{q^t}}
\multisum{k}{p}  
\smallprod{p} \frac{\pqrfac{a_r}{k_r}{q^{h_r}}}{\pqrfac{q^{h_r}}{k_r}{q^{h_r}}}
\frac{\pqrfac{bwq^{t\left(h_1k_1+\cdots+h_pk_p \right)}}{ \infty}{q^t}}
{\pqrfac{wq^{t\left(h_1k_1+\cdots+h_pk_p \right)}}{\infty}{q^t}}
 \smallprod{p} z_r^{k_r} \cr
\cr
& =
\frac{\pqrfac{w}{\infty}{q^t}}{\pqrfac{bw}{\infty}{q^t}}
\multisum{k}{p} 
\smallprod{p} \frac{\pqrfac{a_r}{k_r}{q^{h_r}}}{\pqrfac{q^{h_r}}{k_r}{q^{h_r}}}
\smallprod{p} z_r^{k_r} 
\sum_{j=0}^{\infty} \frac{\pqrfac{b}{j}{q^t}} {\pqrfac{q^t}{j} {q^t}}
w^{j}q^{tj\left(h_1k_1+\cdots+h_pk_p \right)}
  \cr
 &=
\frac{\pqrfac{w}{\infty}{q^t}}{\pqrfac{bw}{\infty}{q^t}}
\multisum{k}{p-1} 
\smallprod{p-1} \frac{\pqrfac{a_r}{k_r}{q^{h_r}}}{\pqrfac{q^{h_r}}{k_r}{q^{h_r}}}
\smallprod{p-1} z_r^{k_r} \cr
&\times \sum_{j=0}^{\infty} \frac{\pqrfac{b}{j}{q^t}} {\pqrfac{q^t}{j} {q^t}} 
w^{j}q^{tj\left(h_1k_1+\cdots+h_{p-1}k_{p-1} \right)} 
 \sum_{k_p\geq 0} 
\frac{\pqrfac{a_p}{k_p}{q^{h_p}}}{\pqrfac{q^{h_p}}{k_p}{q^{h_p}}}
\left(z_pq^{tjh_p}\right)^{k_p}
  \cr
&=
\frac{\pqrfac{w}{\infty}{q^t}}{\pqrfac{bw}{\infty}{q^t}}
 \multisum{k}{p-1} 
\smallprod{p-1} \frac{\pqrfac{a_r}{k_r}{q^{h_r}}}{\pqrfac{q^{h_r}}{k_r}{q^{h_r}}}
\smallprod{p-1} z_r^{k_r} \cr
\times \sum_{j=0}^{\infty} & \frac{\pqrfac{b}{j}{q^t}} {\pqrfac{q^t}{j} {q^t}}
w^{j}q^{tj\left(h_1k_1+\cdots+h_{p-1}k_{p-1} \right)} 
\frac{\pqrfac{a_pz_pq^{tjh_p}}{\infty}{q^{h_p}}}{\pqrfac{z_pq^{tjh_p}}{\infty}{q^{h_p}}}
  \cr
&=
\frac{\pqrfac{w}{\infty}{q^t}}{\pqrfac{bw}{\infty}{q^t}} 
\frac{\pqrfac{a_pz_p}{\infty}{q^{h_p}}}{\pqrfac{z_p}{\infty}{q^{h_p}}} 
\multisum{k}{p-1} 
\smallprod{p-1} \frac{\pqrfac{a_r}{k_r}{q^{h_r}}}{\pqrfac{q^{h_r}}{k_r}{q^{h_r}}}
\smallprod{p-1} z_r^{k_r} \cr
&\times 
\sum_{j=0}^{\infty} \frac{\pqrfac{b}{j}{q^t}} {\pqrfac{q^t}{j} {q^t}}
\frac{\pqrfac{z_p}{tj}{q^{h_p}}}{\pqrfac{a_pz_p}{tj}{q^{h_p}}}
w^{j}q^{tj\left(h_1k_1+\cdots+h_{p-1}k_{p-1} \right)} .
\end{align*}
A careful look at this first step indicates how a multivariable generalization will go. To describe our generalization, we work symbolically. First, represent a (generic) $q$-binomial theorem as
$$\sum_{\k \geq 0} S(\x; n; a, z; q, \k) =P(\x; n; a, z; q),$$
where $\x =(x_1, \dots, x_n)$ and $\k\geq 0$ stands for $k_r\geq 0$, for $r=1, 2, \dots, n$.  
This is used to represent any of the  $q$-binomial theorems, such as \eqref{an-qbin2}, \eqref{an-qbin3}, or indeed any $q$-binomial theorem we have alluded to earlier (or indeed if such a theorem exists). (The symbol $a$ can represent $a_1, \dots, a_n$ as demanded by the context.)

To make our remarks more precise, we isolate the property required for Heine's method to work. Define a {\em Property $H$} as follows. A $q$-binomial theorem of the form $\sum_{\k \geq 0} S(\x; n; a, z; q, \k) =P(\x; a, z; q)$ is said to satisfy  Property $H$ if 
$$S(\x; n; a, zH; q, \k)= H^\sumk S(\x; n; a, z; q, \k).$$

Presumably, our result will have $p$ multi-sums corresponding to the left hand side of \eqref{GB-qlauricella}. So we use the symbols and conventions
\begin{align*}
&\   n_1, n_2, \dots, n_p ;\cr
\x_r &= (x_{r1}, x_{r2}, \dots, x_{rn_r}) \text{ for $r=1, 2, \dots, p$} ;\cr
&\  (\k_1, \k_2, \dots, \k_p) ;\cr
\k_r & = ( k_{r1}, k_{r2}, \dots, k_{rn_r}) \text{ for $r=1, 2, \dots, p$}.
\end{align*}
The $p$ sums will have dimensions $n_1, \dots, n_p$, and we will use the index of summations
$\k_r$ in the $r$th sum. 
We need $p$ copies of $S$ and $P$. Let us say that
$$S_r = S(\x_r; n_r; a_r, z_r; q^{h_r}; \k_r) \text{ and }
P_r(z_r; q^{h_r}) = P(\x_r; n_r; a_r, z_r; q^{h_r}).$$ 
In addition to the above $q$-binomial theorems represented by $(S_r, P_r)$, we require one that we represent as 
$$S_{\text{\it base}} = S(\y; m; b, w; q^{t}; \j) \text{ and } 
P_{\text{\it base}}(w; q^t) =  P(\y; m; b, w; q^{t}).$$
Here the subscript  represents a \lq base' $q$-binomial theorem which we choose from any one of the given options.  There is one more symbol which we will use to represent calculations of the type
$$\frac{1}{\pqrfac{z}{tk}{q^h}}= \frac{\pqrfac{zq^{thk}}{\infty}{q^h}}{\pqrfac{z}{\infty}{q^h}}.$$
We will use
$$Q(z; q^h; k) = P(zq^{hk}; q^h)/P(z; q^h).$$ 
With these symbols, we can now state our generalization of \eqref{GB-qlauricella}.
\begin{Theorem}[A Master theorem] \label{th:an-p-m-GB-qlauricella}
Let $(S_r, P_r)$ for $r=1, \dots, p$, and $(S_{\text{\it base}} , P_{\text{\it base}} )$ each represent any multiple series $q$-binomial theorem satisfying Property $H$, and let the $Q$'s be as above. Then, under suitable convergence conditions, we have
\begin{align}\label{an-p-m-GB-qlauricella}
\multisum{\k}{p} 
{\smallprod{p} S_r } \cdot &{Q_{\text{base}}(w; q^t; h_1 |\k_1|+\cdots+h_p|\k_p|)} \cr
   &= 
\frac{\smallprod{p} P_r(z_r;q^{h_r})}{P_{\text{base}}(w;q^t)}
\sum_{\j \geq 0} S_{\text{base}}
\smallprod{p} {Q_r(z_r; q^{h_r}; t\sumj)}.
\end{align}
\end{Theorem}
\begin{proof} The proof is a straightforward iteration of Heine's method, and is left to the reader.  
\end{proof}
\begin{Remark*} Recall the results of \S \ref{sec:q-euler}. We see that the summands on both sides of \eqref{an-m-kajihara} satisfy Property $H$. So we can intermix $q$-binomial theorems with transformation formulas too. 
\end{Remark*}
Theorem \ref{th:an-p-m-GB-qlauricella} is stated in very general terms and is a master theorem for results in this paper. The results of \S \ref{sec:GB-Heine} are contained in the $p=1$ case of this result. We illustrate the statement by writing down two examples explicitly.

First we take $p=2$, and combine $q$-binomial theorems appearing in the work of Milne and Lilly, Gustafson and Krattenthaler, and one implicit in the author's work with Schlosser. 

We take \eqref{q-bin-GB-MS} as our base $(S_{\text{\it base}}, P_{\text{\it base}}(w; q^t))$. Next let 
$(S_1, P_1)$ from \eqref{an-qbin2}, and $(S_2, P_2)$  from \eqref{an-qbin3}. Now Theorem
\ref{th:an-p-m-GB-qlauricella} reduces to:
\begin{align}
\underset{r=1, \dots, n_1}{\sum\limits_{k_{1r}\geq 0}} \;
\underset{r=1, \dots, n_2}{\sum\limits_{k_{2r}\geq 0}}
\Bigg( & 
\triprod{n_1} \!
\frac{1-q^{h_1(k_{1r}-k_{1s})}x_{1r}/x_{1s} }{1-x_{1r}/x_{1s} }
\triprod{n_2} \!
\frac{1-q^{h_2(k_{2r}-k_{2s})}x_{2r}/x_{2s} }{1-x_{2r}/x_{2s} }\cr
&\times
\sqprod{n_1} \frac{\pqrfac{a_{1s} x_{1r}/x_{1s}}{k_{1r}}{q^{h_1}}}
 {\pqrfac{q^{h_1}x_{1r}/x_{1s}}{k_{1r}}{q^{h_1}}}
\smallprod{n_2}  \frac{\pqrfac{a_{2} }{k_{2r}}{q^{h_2}}}
{\pqrfac{q^{h_2}}{k_{2r}}{q^{h_2}} } \cr
&\times
\frac{\pqrfac{w}{h_1|\k_1|+h_2|\k_2|}{q^t}}
{\pqrfac{b_1\cdots b_mw}{h_1|\k_1|+h_2|\k_2|}{q^t}}
z_1^{|\k_1|}z_2^{|\k_2|} 
\notag \\
&\times 
q^{h_1\suml_{r=1}^{n_1} (r-1)k_{1r} +h_2\suml_{r=1}^{n_2} (r-1)k_{2r} }
  \htqelementary{k_1}{h_1} \smallprod{n} x_{1r}^{-k_{1r}}
\Bigg) \notag \\
= \smallprod{n_1} & \frac{\pqrfac{a_{1r}z_1/x_{1r}}{\infty}{q^{h_1}}}
{\pqrfac{z_1/x_{1r}}{\infty}{q^{h_1}}} 
\smallprod{n_2}
\frac{\pqrfac{a_2z_2q^{h_2(r-1)}}{\infty}{q^{h_2}}}{\pqrfac{z_2q^{h_2(r-1)}}{\infty}{q^{h_2}}} 
\cdot
\frac{\pqrfac{w}{\infty}{q^t}}{\pqrfac{b_1\cdots b_m w }{\infty}{q^t}}
 \notag \\
\times \multisum{j}{m} & \Bigg( \htvandermonde{y}{j}{m}{t}
\sqprod{m} \frac{\pqrfac{b_sy_r/y_s }{j_r}{q^t}}{\pqrfac{q^t\xover{y}}{j_r}{q^t} } 
\cr
&\times \smallprod m 
\frac{\pqrfac{cy_r/b_1\cdots b_m}{j_r}{q^t} \pqrfac{cy_r}{\sumj}{q^t} }
{\pqrfac{cy_r}{j_r}{q^t} \pqrfac{cy_r/b_r}{\sumj}{q^t}} 
\cdot w^{\sumj} \htpowerq{j}{m}{t} 
\notag \\
& \times  
 \smallprod{n_1} \frac{\pqrfac{z_1/x_{1r}}{t\sumj}{q^{h_1}}}
 {\pqrfac{a_{1r}z_1/x_{1r}}{t\sumj}{q^{h_1}}}
 \cdot
  \smallprod{n_2} \frac{\pqrfac{z_2q^{h_2(r-1)}}{t\sumj}{q^{h_2}}}
  {\pqrfac{a_2z_2q^{h_2(r-1)}}{t\sumj}{q^{h_2}}}
\Bigg). \label{an-m-GB-qlauricella-big}
\end{align}
Here, for convergence, we require $|q^{h_r}|<1,$ $|q^{t}|<1,$ $|q^{th_r}|<1$ for $r=1, 2$; $|w|<1$, and  $|z|<1$.   

In our second example of an extension of \eqref{GB-qlauricella}, we take  $p\mapsto p+1$, 
$(S_r, P_r)$ (for $r=1, 2, \dots, p$)
 to all be the $1$-variable $q$-binomial theorem \eqref{q-bin}. We take $h_1=h_2=\cdots =h_p=h$. For $(S_{p+1}, P_{p+1})$ and the base $(S_{\text{\it base}}, P_{\text{\it base}}(w; q^t))$, we take the $c=0=d$ case of \eqref{q-bin-GB-MS} with bases $q^h$ and $q^t$, respectively. In this manner we obtain, assuming the usual convergence conditions, $|z|<1$, $|w|<1$, and $|u_r|<1$ (for $r=1, 2, \dots, p$):
\begin{align}
\multisum{l}{p}\; \multisum{k}{n} \Bigg( & \hvandermonde{x}{k}{n} 
\sqprod{n} \frac{\pqrfac{a_s\xover{x}}{k_r}{q^h} }{\pqrfac{q^h\xover{x}}{k_r}{q^h} } \notag \\
&\times \smallprod{p} \frac{\pqrfac{c_r}{l_r}{q^{h}}}{\pqrfac{q^{h}}{l_r}{q^{h}}} \cdot
\frac{\pqrfac{w}{h(\sumk+\sumvecl)}{q^t}}
{\pqrfac{b_1b_2\cdots b_mw}{h(\sumk+\sumvecl )}{q^t}} 
 \smallprod{p} u_r^{l_r}  \cdot  z^{\sumk} \hpowerq{k}{n} \Bigg) \notag\\
= & \frac{\pqrfac{w}{\infty}{q^t} \pqrfac{a_1a_2\cdots a_n z}{\infty}{q^h}}
{\pqrfac{b_1b_2\cdots b_mw }{\infty}{q^t} \pqrfac{z}{\infty}{q^h}}
\smallprod{p} \frac{\pqrfac{c_ru_r}{\infty}{q^h}}{\pqrfac{u_r}{\infty}{q^h}}
\notag \\
\times \multisum{j}{m} & \Bigg( \htvandermonde{y}{j}{m}{t}
\sqprod{m} \frac{\pqrfac{b_s\xover{y}}{j_r}{q^t} }{\pqrfac{q^t\xover{y}}{j_r} {q^t}} \notag \\
& \times  \frac{\pqrfac{z}{t\sumj}{q^h}}{\pqrfac{a_1a_2\cdots a_nz}{t\sumj}{q^h}}
\smallprod{p} \frac{\pqrfac{u_r}{t\sumj}{q^h}}{\pqrfac{c_ru_r}{t\sumj}{q^h}}
\cdot w^{\sumj} \htpowerq{j}{m}{t} 
\Bigg). \label{an-m-lauricella1}
\end{align}
Note that we have used product of $n$ $q$-binomial theorems to deal with a product of the form
$$
\smallprod{n} \frac{\pqrfac{a_rx_r}{\infty}{q}}{\pqrfac{x_r}{\infty}{q}}
$$
So the list at the end of \S \ref{sec:GB-Heine} is not a complete list. 

\section{Summary and credits}\label{sec:summary}

The objective of this paper is to study Heine's method as applied to series of type $A$. It is a continuation of our work \cite{GB2017}, which arose while studying Andrews and Berndt~\cite[Ch.~ 1]{AB2009}. We conclude by summarizing our work and pointing out related ideas in the literature. 

In \S \ref{sec:GB-Heine}, we gave some examples of multiple series generalizations of the bibasic Heine transformation. These (and other similar formulas) can be obtained using the $q$-binomial theorems in \cite{Milne1997, ML1995, GK1996, Ka2012, BS2018a}. 
Our results transform an $n$-dimensional series to a multiple of an $m$-dimensional series. Early examples of such identities were given by  Gessel and Krattenthaler \cite{GeK1997}, Kajihara \cite{Ka2004} and Kajihara and Noumi \cite{KN2003}.

Previously, Gustafson and Krattenthaler \cite{GK1996, GK1997} gave $A_n$ Heine transformation formulas.  However, Heine's method is not applicable to obtain their results. 
It would be interesting to see whether their work extends to generalize the bibasic Heine transformation, and examine the generalizations of Ramanujan's $_2\phi_1$ transformations (if any), as has been done in \S \ref{sec:ramanujan-2p1}. 

 Section \ref{sec:ramanujan-2p1} contains extensions of some of Ramanujan's $_2\phi_1$ transformations from Andrews and Berndt \cite[Ch.\ 1]{AB2009}. In this regard, the usefulness of the bibasic Heine transformation was pointed out in \cite{GB2017}. These special cases are chosen to bring the multivariable identities close to Ramanujan's own identities. They are motivated by previous work of, for example, Milne \cite{Milne1992} and Krattenthaler \cite{Krat2001} (who proved a conjecture of Warnaar). We remark that there is  much more in Andrews' work \cite{An1966a} (where Heine's method is examined in detail) which we expect to study on another occasion.

In \S \ref{sec:q-euler} we demonstrate that Heine's method can be applied to transformation formulas too, and gain some understanding of when this method applies. This leads to the master theorem of \S \ref{sec:qlauricella}, which follows by iterating Heine's method. This extends the bibasic version of Andrews' transformation formula for the $q$-Lauricella function given 
by Agarwal, Jain and Choi \cite{AJC2017}, and independently, by the author \cite{GB2017}. 
Previously,  Milne~\cite[\S 7]{Milne1997} gave such results, but we feel our approach and results are truer to Andrews' original formulation of his results.  

We have already noted that the master theorem applies to using transformation formulas as in \S \ref{sec:q-euler}.  In addition, we mention that the multivariable $q$-binomial theorem of Macdonald \cite{Mac2013} (see Kaneko \cite[Th.\ 3.5]{JK1996}) is also amenable to Heine's method, and can be intermixed with above mentioned $q$-binomial theorems.  The Property $H$ follows from the homogeneity of the Macdonald polynomials. If we apply Heine's method to Macdonald's $q$-binomial theorem, we obtain bibasic Heine transformation formulas different in character 
from, but perhaps not as deep as, the Heine transformation formula given by Baker and Forrester \cite{BF1999}. However, what is interesting in our approach is that we can intermix the two kinds of series. See also Warnaar \cite[Th.\ 2.3]{SOW2010} for a multivariable, dimension changing, formula---which extends Heine's second transformation formula \cite[Eq.\ (1.4.5)]{GR90}.

\subsection*{Acknowledgements} We thank the anonymous referee for helpful suggestions. 



\end{document}